\documentclass{llncs}
\usepackage{amsmath}
\usepackage{amssymb}
\usepackage{enumerate}
\usepackage{varioref}
\usepackage{charter,eulervm}

\newcommand{\es}{\varnothing}

\title{{\sc Feedback Vertex Set on Chordal Bipartite Graphs}}

\author{
 Ton~Kloks%
\thanks{This author is supported by the
National Science Council of Taiwan, under grant
NSC~99--2218--E--007--016.}
\and
Ching-Hao~Liu\inst{} 
\and 
Sheung-Hung~Poon\inst{}
}
\institute{
 Department of Computer Science\\
 National Tsing Hua University,
 No.~101, Sec.~2, Kuang Fu rd., Hsinchu, Taiwan\\
 {\tt chinghao.liu@gmail.com, spoon@cs.nthu.edu.tw}
}

\begin{document}

\maketitle

\begin{abstract}
Let $G=(A,B,E)$ be a bipartite graph with color classes $A$ and $B$. 
The graph $G$ is chordal bipartite if $G$ has no induced cycle of 
length more than four. Let $G=(V,E)$ be a graph. A feedback vertex 
set $F$ is a set of vertices $F \subset V$ such that $G-F$ is a forest. 
The feedback vertex set problem asks for a feedback vertex set of 
minimal cardinality. We show that the feedback vertex set problem 
can be solved in polynomial time on chordal bipartite graphs. 
\end{abstract}

\section{Introduction}

The feedback vertex set problem asks for a 
minimum set of vertices that meets all cycles of a 
graph. 
The feedback vertex set problem is a benchmark problem for  
fixed-parameter algorithms, exact algorithms, approximation 
algorithms, and for 
algorithms on special graph classes~\cite{kn:festa}. In this paper 
we show that the feedback vertex set problem can be solved in 
polynomial time for bipartite graphs without chordless cycles of length 
more than four.  

The feedback vertex set problem can 
be solved in time $O(1.7548^n)$~\cite{kn:fomin}. 
If a graph $G$ 
has a feedback vertex set with at most $k$ vertices then so does 
any minor of $G$. It thus follows from the 
graph minor theorem that the problem is fixed-parameter 
tractable~\cite{kn:niedermeier}. 
It is possible 
to reduce $k$-feedback vertex set to a quadratic kernel in 
polynomial time~\cite{kn:thomasse}. At the moment, the best algorithm 
for $k$-feedback vertex set seems to 
run in time $O(3.83^k kn^2)$~\cite{kn:cao}.  

The feedback vertex set problem is 
NP-complete~\cite{kn:garey} and remains so on bipartite graphs and 
on planar graphs. Recently, it was shown that the 
problem remains NP-complete on tree convex bipartite graphs~\cite{kn:wang}. 
There is a factor two approximation 
algorithm~\cite{kn:becker}.  

Note that the problem can be formulated in monadic second-order logic without 
quantification over subsets of edges and it follows  
that the problem can be solved in $O(n^3)$ time for 
graphs of bounded treewidth or rankwidth~\cite{kn:courcelle}. 
The problem can also be solved in polynomial time on, {\em e.g.\/},  
interval graphs, chordal graphs, permutation graphs,  
cocomparability graphs, convex bipartite graphs, and 
AT-free graphs%
~\cite{kn:kratsch,kn:liang,kn:liang2,kn:lu}. For information 
on various graph classes we refer 
to~\cite{kn:brandstadt,kn:golumbic2,kn:mckee,kn:spinrad}.   
 
\section{Preliminaries}

If $A$ and $B$ are sets then we use $A+B$ and $A-B$ to denote 
$A \cup B$ and $A \setminus B$ respectively. For a set $A$ and an 
element $x$ we also write $A + x$ and $A - x$ instead of 
$A+ \{x\}$ and $A - \{x\}$. 
 
A graph is a pair $G=(V,E)$ where $V$ is a finite, nonempty set and
where $E$ is a set of two-element subsets of $V$. We call the
elements of $V$ the vertices or points of the graph. We denote the
elements of $E$ as $(x,y)$ where $x$ and $y$ are vertices. We call
the elements of $E$ the edges of the graph. If $e=(x,y)$ is an edge
of a graph then we call $x$ and $y$ the endpoints of $e$ and we say
that $x$ and $y$ are adjacent. The neighborhood of a vertex $x$ is
the set of vertices $y$ such that $(x,y) \in E$. We denote this
neighborhood by $N(x)$. The closed neighborhood of a vertex 
$x$ is defined as $N[x]=N(x)+x$. The degree of a vertex $x$ is the
cardinality of $N(x)$. Let $W \subseteq V$ and let $W \neq \es$. The
graph $G[W]$ induced by $W$ has $W$ as its set of vertices and it
has those edges of $E$ that have both endpoints in $W$. If $W
\subset V$, $W \neq V$, then we write $G-W$ for the graph induced by
$V \setminus W$. If $W$ consists of a single vertex $x$ then we
write $G-x$ instead of $G-\{x\}$. We usually denote the number of
vertices of a graph by $n$ and the number of edges of a graph by $m$.

A path is a graph of which the vertices can be linearly
ordered such that the pairs of consecutive vertices form the
set of edges of the graph. We call the first and last vertex in this
ordering the terminal vertices of the path.
We denote a path with $n$ vertices by $P_n$. To denote a specific
ordering of the vertices we use the notation $[x_1,\ldots,x_n]$.
Let $G$ be a graph. Two vertices $x$ and $y$ of $G$ are connected by a
path if $G$ has an induced subgraph which is a path with $x$
and $y$ as terminals. Being connected by a path is an equivalence relation
on the set of vertices of the graph.
The equivalence classes are called the
components of $G$.
A cycle consists of a path with at least three vertices with
one additional edge that connects the two terminals of the path.
We denote a cycle with $n$ vertices by $C_n$.
The length of a path or a cycle is its number of edges.

A clique in a graph is a nonempty subset of
vertices such that every pair in it is adjacent.
An independent set in a graph is a nonempty subset of
vertices with no edges between them.
A graph $G=(V,E)$ is bipartite if there is a partition $\{A,B\}$ 
of the set of vertices $V$ 
in two independent sets. One part of the partition may be empty. 
We use the notation $G=(A,B,E)$ to denote a bipartite graph with 
independent sets $A$ and $B$ and we call $A$ and $B$ the 
color classes of $G$. 
A bipartite graph $G=(A,B,E)$ is complete bipartite if any two vertices
in different color classes are adjacent. 

\begin{definition}
A bipartite graph $G=(A,B,E)$ is {\em chordal bipartite\/} if 
it has no induced cycle of length more than four. 
\end{definition}

Chordal bipartite graphs were introduced in~\cite{kn:golumbic}. 
They are characterized by the property of having a `perfect edge 
without vertex elimination ordering.' 

\begin{definition} 
Let $G=(A,B,E)$ be a bipartite 
graph. An edge $e=(x,y)$ is {\em bisimplicial\/} if 
$N(x) \cup N(y)$ induces a complete bipartite subgraph in $G$. 
\end{definition}

Let $G=(A,B,E)$ be a bipartite graph and let $(e_1,\ldots,e_m)$  
be an ordering 
of the edges of $G$. Define $G_0=G$, $E_0=E$ and 
for $i=1,\ldots,m$ define $G_i=(A,B,E_i)$ as the graph 
with $E_i=E_{i-1}-\{e_i\}$. 
Thus $G_i$ is obtained from $G_{i-1}$ by removing the edge $e_i$ but not 
its endvertices. 
The ordering $(e_1,\ldots,e_m)$ is a perfect edge without vertex elimination 
ordering if $e_i$ is bisimplicial in $G_{i-1}$ for $i \in \{1,\ldots,m\}$. 

\begin{theorem}[\cite{kn:brandstadt,kn:kloks}]
\label{char}
A bipartite graph is chordal bipartite if and only if it has a perfect 
edge without vertex elimination ordering.
\end{theorem}

\noindent
If $G$ is chordal bipartite then any bisimplicial edge can start 
a perfect edge without vertex elimination ordering. Thus Theorem~\ref{char}  
provides a greedy 
recognition algorithm for chordal bipartite 
graphs~\cite{kn:brandstadt,kn:kloks,kn:kloks3}.  

\bigskip
We also need the following characterization of chordal bipartite graphs. 

A graph is chordal if it has no induced cycle of length more 
than three~\cite{kn:hajnal}. A famous theorem of Dirac shows that a 
graph is chordal if and only if every minimal separator is a 
clique~\cite{kn:dirac} (see the next section for the definition 
of a minimal separator). 
A vertex is simplicial if its neighborhood 
induces a clique. A graph is chordal if and only if every induced 
subgraph has a simplicial vertex~\cite{kn:dirac}. Thus the graph 
has a `perfect elimination ordering' of the vertices by removing 
simplicial vertices from the graph one by one.     

Let $k \geq 3$. A $k$-sun is a chordal graph with $2k$ vertices, 
partitioned into two ordered sets $C=\{c_1,\ldots,c_k\}$ and 
$S=\{s_1,\ldots,s_k\}$. The graph induced by $C$ is a clique 
and, for $i\in \{1,\ldots,n\}$, 
$s_i$ is adjacent to $c_i$ and $c_{i+1}$ where $c_{n+1}=c_1$. 
A graph is strongly chordal if it is chordal and has no induced sun. 
The structure of strongly chordal graphs was first analyzed by Farber 
at about the same time during which 
the class of chordal bipartite graphs, as 
`totally balanced matrices,'  
was investigated by various other authors. 
It turns out that the two classes of graphs 
have, essentially the same structure. 

Consider a bipartite graph $G=(A,B,E)$. Let $G_A$ be the 
chordal graph obtained 
from $G$ by adding an edge between every pair of vertices in $A$. 

\begin{theorem}[\cite{kn:dahlhaus}] 
\label{dahlhaus}
A bipartite graph $G=(A,B,E)$ is chordal bipartite if and only if 
$G_A$ is strongly chordal. 
\end{theorem}

\noindent
Strongly chordal graphs have a special perfect elimination ordering of its 
vertices. 

\begin{definition}
Let $G$ be a graph. A vertex $x$ is {\em simple\/} if the vertices 
of $N[x]$ can be ordered $x_1,\ldots,x_{\ell}$ such that, 
for $i=1,\ldots,\ell-1$,   
$N[x_i] \subseteq N[x_{i+1}]$. 
\end{definition}

\begin{theorem}
[\cite{kn:anstee2,kn:anstee,kn:brouwer2,kn:farber,kn:kolen,kn:lehel}]%
\label{simple vertex}
A graph is strongly chordal if and only if every induced subgraph 
has a simple vertex. 
\end{theorem}

\noindent
Let $G=(A,B,E)$ be chordal bipartite. 
It is easy to see that if $B \neq \es$, then $B$ has a vertex that is 
simple in $G_A$ (see, {\em e.g.\/},~\cite{kn:hammer,kn:huang}). 

The ordering of the vertices of a strongly 
chordal graph, obtained by repeatedly removing a simple 
vertex is called a `simple elimination ordering.' 
     
\bigskip
More information on the structure of chordal bipartite graphs  
can be found 
in the 
books~\cite{kn:berge,kn:brandstadt,kn:golumbic2,kn:mckee,kn:spinrad} 
and in various papers, {\em e.g.\/},~\cite{kn:anstee}. 
We refer to the literature  
for many different kinds of characterizations. 
Chordal bipartite graphs 
can be recognized in $O(n^2)$ time~\cite{kn:kloks3,kn:spinrad2}. 
Several NP-complete problems 
can be solved in polynomial time on chordal bipartite graphs.  
Others, such as the domination problem, remain 
NP-complete~\cite{kn:damaschke,kn:muller}.      

\begin{definition}
Let $G=(V,E)$ be a graph. 
A set $F \subset V$ is a {\em feedback vertex set\/} 
of $G$ if $G-F$ has no induced cycle. 
\end{definition}

\noindent
The feedback vertex set problem asks for a feedback vertex set 
of minimal cardinality.   
In this paper we show that the feedback vertex set problem can be 
solved in polynomial time on chordal bipartite graphs. 

\section{Separators in chordal bipartite graphs}

In this section we analyze the structure of chordal bipartite 
by means of minimal separators. 

\begin{definition}
Let $G=(V,E)$ be a graph. A set $S \subset V$ is a {\em separator\/} 
of $G$ if $G-S$ has at least two components. 
\end{definition}

\noindent
If every vertex of a separator $S$ has a neighbor in a component 
$C$ of $G-S$ then we say that $C$ is close to $S$.  
A separator $S$ is {\em minimal\/} if $G-S$ has two components  
that are close to $S$. 

\begin{remark} 
A similar concept was introduced 
in~\cite{kn:golumbic}. In this paper the authors 
prove that a graph is chordal bipartite if and only if every `minimal 
edge separator' is complete bipartite.  Here, a minimal edge separator is a 
minimal separator which separates two edges with nonadjacent 
endpoints into distinct components.   
\end{remark}

The following lemma generalizes the result of~\cite{kn:golumbic}. 
  
\begin{lemma}
\label{min sep is CB}
Let $G=(A,B,E)$ be chordal bipartite. Let $S$ be a minimal separator of 
$G$. Then $G[S]$ is complete bipartite. 
\end{lemma}
\begin{proof}
Assume that $S$ has two nonadjacent vertices $x \in A$ and $y \in B$. 
Since $S$ is minimal there are two components $C_1$ and $C_2$ in $G-S$ 
that are close to $S$. For $i \in \{1,2\}$ let $p_i$ be a neighbor 
of $x$ in $C_i$ and let $q_i$ be a neighbor of $y$ in $C_i$. 
Then $p_i \in B$ and $q_i \in A$ since $G$ is bipartite. Since $G[C_i]$ 
is connected there exist a path $P_i$ in $G[C_i]$ with terminals $p_i$ and 
$q_i$. We may choose $p_i$ and $q_i$ such that $x$ and $y$ have no other 
neighbors than $p_i$ and $q_i$ in $P_i$. Since $C_1$ and $C_2$ are components 
of $G-S$ no vertex of $P_1$ is adjacent to any vertex of $P_2$. It follows 
that $P_1+P_2+\{x,y\}$ induces a cycle of length at least 6 which is a 
contradiction. 

This proves the lemma.
\qed\end{proof}

\begin{lemma}
\label{simplicial} 
Let $G=(A,B,E)$ be chordal bipartite. Let $S$ be a minimal 
separator of $G$ and let $C$ be a component of $G-S$ that is 
close to $S$. If $S \cap A \neq \es$ then there exists a 
vertex $x$ in $C$ with 
\[N(x) \cap S = S \cap A.\]
\end{lemma}
\begin{proof}
Let $C^{\prime} \neq C$ be another component of $G-S$ that is close 
to $S$. Consider the subgraph induced by $S+C_1+C_2$ For 
simplicity, denote this subgraph also by $G$. Assume that 
$S \cap A \neq \es$. Then $E \neq \es$ otherwise $S$ is not a 
minimal separator. 

If 
$C$ contains only one vertex $x$ then $x$ is adjacent to 
all of $S$. Then $S \cap B =\es$ and $N(x) \cap S = S \cap A$ 
because $G$ is bipartite and $C$ is close to $S$. Assume that $C$ 
contains at least two vertices. 

Since $G$ is chordal bipartite it has a 
bisimplicial edge $e=(p,q)$. Assume that $p \in A$ and $q \in B$. 
We consider the following cases. 

Assume that $p$ and $q$ 
are both in $S$. Because $C$ and $C^{\prime}$ are close to $S$ 
$p$ has a neighbor in $C$ and $q$ has a neighbor in $C^{\prime}$. 
This is a contradiction since no vertex of $C$ is adjacent to 
any vertex of $C^{\prime}$ and thus $e$ is not bisimplicial. 

Assume that $p \in S$ and $q \in C$. Since $C$ has at least two vertices 
$q$ has a neighbor in $C$  and since $C^{\prime}$ is close to $S$ 
$p$ has a neighbor in $C^{\prime}$. This contradicts the assumption 
that $e$ is bisimplicial. 

Assume that $p \in S$ and $q \in C^{\prime}$. 
By the previous argument $C^{\prime}=\{q\}$. Also $S \cap B =\es$ 
since every vertex of $S$ has a neighbor in $C^{\prime}$. Let $x$ 
be any neighbor of $p$ in $C$. Since $(p,q)$ is bisimplicial, and 
$N(q)=S$, $x$ is adjacent to all vertices of $S$ as well. 

Assume 
that $p$ and $q$ are both in $C^{\prime}$. Let $G^{\prime}$ be the 
graph obtained from $G$ by removing the edge $e$ but not its endpoints. 
Then $G^{\prime}$ is chordal bipartite. If the component $C^{\prime}$ 
remains connected after removal of $e$ we can use induction since 
$S$ is a minimal separator of $G^{\prime}$, $C$ is a component of 
$G^{\prime}-S$ that is close to $S$ and $G^{\prime}$ has fewer 
edges than $G$. 

Assume that the removal of $e$ disconnects $C^{\prime}$. 
Then one of $p$ and $q$ has only neighbors in $S$. 
First assume that $p$ 
has only neighbors in $S$.  
Let $D$ be the component of $G^{\prime}-S$ that contains $q$. 
If every vertex of $N(p) \cap S$ has a neighbor 
in $D$ then $D$ is close to $S$ in $G^{\prime}$ 
and we can apply induction as above. 
Assume that some vertex $s \in S$ has $N(s) \cap C^{\prime} =\{p\}$. If 
$q$ has a neighbor $q^{\prime}$ in $D$ then $s$ is adjacent to $q^{\prime}$ 
which is a contradiction. 
Thus   
$C^{\prime}=\{p,q\}$, $N(p) \cap S =S \cap B$, and 
$N(q) \cap S = S \cap A$. 
Now $S \cap A$ is a minimal separator of $G^{\prime}$ with 
close components $C + (S \cap B) + \{p\}$ and $\{q\}$. 
Let 
\[G^{\prime\prime}=G^{\prime}-((S\cap B) \cup \{p\}).\]
Then $S \cap A$ is a minimal separator in $G^{\prime\prime}$ with 
close components $C$ and $\{q\}$. By induction $C$ has a vertex 
adjacent to all vertices of $S \cap A$. 

Now assume that $q$ has only neighbors in $S \cap A$ in $G^{\prime}$. 
Let $D$ be the component of $G^{\prime}-S$ that contains $p$. Assume 
there exists a vertex $s \in S$ with $N(s) \cap C^{\prime} = \{q\}$. 
If $p$ has a neighbor $p^{\prime} \in D$ then $s$ is 
adjacent to $p^{\prime}$ 
which is a contradiction. Thus in this case $C^{\prime} = \{p,q\}$ 
and we obtain the result as above. Otherwise, every vertex of 
$N(q) \cap S$ has a neighbor in $D$. Then $S$ is a minimal separator 
in $G^{\prime}$ and we can use induction.  
  
Assume that $p$ and $q$ are both in $C$. If the removal of the edge 
$e=(p,q)$ does not disconnect $C$ then 
$S$ is a minimal separator in the graph $G^{\prime}$, 
obtained by removing 
$e$, and $C^{\prime}$ and $C$ are close to $S$ in $G^{\prime}$. 
Otherwise, it follows as in the case where $p$ and $q$ are both in 
$C^{\prime}$ that 
either we can apply induction on a component 
$D \subset C$ of $G^{\prime}-S$ or,     
\[C=\{p,q\} \quad\text{and}\quad N(p) \cap S=S \cap B
\quad\text{and}\quad N(q) \cap S = S \cap A.\]

This proves the lemma. 
\qed\end{proof}

\begin{lemma}
\label{C has bisimplicial}
Let $G=(A,B,E)$ be chordal bipartite. Let $S$ be a minimal separator 
and let $C$ be a component of $G-S$ that is close to $S$. 
If $C$ has an edge then it has an edge which is bisimplicial in $G$. 
\end{lemma}
\begin{proof}
First notice that the claim holds true when $C$ has 
only two vertices; in that case, by Lemma~\ref{min sep is CB}, 
$S+C$ induces a complete bipartite graph 
and the edge in $C$ is bisimplicial. 
If $S=\es$ then $C$ is a 
component of $G$ and the claim follows since $G[C]$ has a bisimplicial 
edge and this edge is also bisimplicial in $G$.  

Let $C^{\prime}$ be another component of $G-S$ that is close to $S$. 
Consider the subgraph induced by $S+C+C^{\prime}$. It suffices to 
prove the claim for this induced subgraph. For simplicity we call 
this induced subgraph 
also $G$. Let $(p,q)$ be a bisimplicial edge in $G$ with 
$p \in A$ and $q \in B$. We consider the following cases. 

Assume that $p \in S$ and that $q \in C$. The vertex $q$ has a neighbor $x$  
in $C$ since $|C| > 1$ and $G[C]$ 
is connected. Since $S$ is minimal, $p$ has a neighbor 
$y \in C^{\prime}$. This is a contradiction since $x$ and $y$ are 
not adjacent. 

Assume that $p \in S$ and that $q \in S$. The vertex $p$ has a neighbor 
$x \in C$ and the vertex $q$ has a neighbor $y \in C^{\prime}$ since 
$C$ and $C^{\prime}$ are close to $S$. This is a contradiction 
since $x$ and $y$ are not adjacent. 

Assume that $p \in S$ and that $q \in C^{\prime}$. If $q$ has 
a neighbor $y \in C^{\prime}$ then we derive a contradiction as above. 
Thus $C^{\prime}=\{q\}$ and $N(q)=S$ and $S \cap B = \es$. Let 
$\Omega = N(p) \cap C$. Then every vertex of $\Omega$ is adjacent to 
every vertex of $S$ since $(p,q)$ is bisimplicial. 
Let $C_1,\ldots,C_t$ be the components 
of $G[C]-\Omega$. Assume that a component, say $C_1$ has at least 
two vertices. Let $S^{\prime} \subseteq S+\Omega$ be the set of 
vertices with a neighbor in $C_1$. We claim that $S^{\prime}$ is a 
minimal separator. First notice that $C_1$ is a component of 
$G-S^{\prime}$ and that every vertex of $S^{\prime}$ has a neighbor in 
$C_1$ by construction. Also, $p$ and $q$ are not 
in $S^{\prime}$ since they have no neighbors in $C_1$. 
Let $C^{\prime\prime}$ be the component 
of $G-S^{\prime}$ that contains the edge $(p,q)$. Then every vertex of 
$S^{\prime}$ has a neighbor in $C^{\prime\prime}$ since it is 
adjacent to $p$ or to $q$. We can now use induction on the number 
of vertices in the component $C$ and conclude that 
$C_1$ has an edge which is bisimplicial in $G$. 

Assume that every component of $G[C]-\Omega$ has only one vertex. 
Notice that $C-\Omega$ has only vertices in $A$ since $G[C]$ is connected 
and $\Omega \subseteq N(p) \subseteq B$ since $G$ is bipartite 
and $p \in A$. 
Consider the vertices of $C \cap A$ and $\Omega$.   
Note that $C \cap A \neq \es$ since $|C| > 1$ and $G[C]$ is connected 
and bipartite. 
The graph induced by $(C \cap A) + \Omega$ 
has a bisimplicial edge. This edge is also bisimplicial in $G$ since 
every vertex of $\Omega$ is adjacent to every vertex of $S$. 

Assume that $p \in C^{\prime}$ and that $q \in C^{\prime}$. 
If the removal of the edge $(p,q)$ from the graph 
leaves $C^{\prime}$ connected then the 
claim follows by induction on the number of edges in $G$. 
Otherwise, after removal of the edge $(p,q)$, one of $p$ and 
$q$ has only neighbors in $S$. 
Say $p$ has only neighbors in $S$. 
Let $G^{\prime}$ be the graph obtained from $G$ by removing the edge $(p,q)$ 
and let $D$ be the component of $G^{\prime}-S$ that contains $q$. 
If every vertex of $S$ has a neighbor in $D$ then $S$ is a minimal 
separator in $G^{\prime}$ with close components $C$ and $D$. In 
that case we can proceed by 
induction on the number of edges in $G$. 
Assume some vertex $s \in S$ has no neighbors in $D$ Then $s$ is 
adjacent to $p$ since $C^{\prime}$ is close to $S$. 
If $q$ has a neighbor in $D$ then we arrive at a contradiction 
since $s$ is adjacent to this neighbor. 
We may now conclude that $C^{\prime}=\{p,q\}$.    

The graph $G^{\prime}$ is chordal bipartite. Thus it has a 
bisimplicial edge $(a,b)$ with $a \in A$ and $b \in B$. Assume that 
$a=p$. Then $b \in S$. Let $\Gamma \subseteq C$ be the set of 
neighbors of $b$ in $C$. Then every vertex of $\Gamma$ is 
adjacent to every vertex of $S \cap B$ since every vertex of 
$S \cap B$ is adjacent to $a$ and $(a,b)$ is bisimplicial. 
Consider the components $O_1,\ldots,O_{\ell}$ of $G[C]-\Gamma$. 
Assume that $|O_1| > 1$. Let $S^{\prime} \subseteq S+\Gamma$ be the 
subset that has neighbors in $O_1$. Then $S^{\prime}$ is a minimal 
separator in $G$. Since 
$|O_1| < |C|$ we can use induction 
and conclude that $G[O_1]$ has an edge which is bisimplicial 
in $G$. Assume that all components $O_i$ have only one vertex. 
Then $C-\Gamma$ has only vertices in $B$ since $C$ is connected and 
$\Gamma \subseteq A$. Also, $C-\Gamma \neq \es$ since $C$ is 
connected and $|C| > 1$.   
Consider the subgraph $G^{\prime\prime}$ 
of $G$ induced by $C+(S \cap A) + q$. Notice that $G^{\prime\prime}$ 
is connected and that $S \cap A$ is a minimal separator 
in $G^{\prime\prime}$ with close components $C$ and $\{q\}$. 
By induction on the number of vertices 
in $G$ we may conclude that $C$ has an edge which is 
bisimplicial in $G^{\prime\prime}$. This edge is also bisimplicial 
in $G$ which follows from the fact that  
every vertex of $\Gamma$ is adjacent to every vertex 
of $S \cap B$.  
 
Assume that $a \neq p$ and that $b \neq q$. Assume that $a \in S$ and 
$b \in S$. Then $a$ has a neighbor in $C$ and $b$ is adjacent to 
$p$, which is a contradiction. Assume that $a \in S$ and that 
$b \in C$. Since $|C| > 1$ and $C$ is connected, 
$b$ has a neighbor $x \in C$. Then we obtain a contradiction since 
$a$ is adjacent to $q$ and $(x,q) \not\in E$. 

We conclude that $a \in C$ and that $b \in C$. Then $(a,b)$ is 
a bisimplicial edge in the graph $G$ since $a$ and $b$ have only 
neighbors in $C+S$. 
   
This proves the lemma.  
\qed\end{proof}

\begin{corollary}
Let $G=(A,B,E)$ be chordal bipartite and let $S$ be a 
minimal separator of $G$. Let $C$ be a component of $G-S$ 
that is close to $S$. Assume that $S \cap A \neq \es$ and 
that $S \cap B \neq \es$. Then there exist adjacent vertices 
$x$ and $y$ in $C$ such that 
\[N(x) \cap S = S \cap A \quad\text{and}\quad N(y) \cap S = S \cap B.\] 
\end{corollary}

\begin{corollary}
Let $G$ be chordal bipartite. 
The number of minimal separators in $G$  
is $O(n+m)$. 
\end{corollary}

\begin{remark}
A somewhat weaker upperbound for the number of minimal 
separators in a chordal bipartite graph, 
{\em i.e.\/}, 
$O\left(n+\binom{m}{2}\right)$, 
was obtained in~\cite{kn:kratsch2} (see also~\cite{kn:brandstadt}). 
Note that there exists an algorithm with polynomial delay that 
lists all the minimal separators 
of an arbitrary graph%
~\cite{kn:kloks2}. 
\end{remark}

\section{Maximal chordal bipartite graphs}
\label{maximal CBG}

Let $k$ be a natural number. A $k$-tree is a chordal graph, defined 
recursively as follows~\cite{kn:beineke}. 
A $k$-tree on $k+1$ vertices is a clique. Given a $k$-tree $G$ with 
$n$ vertices, one can obtain a $k$-tree with $n+1$ vertices by 
introducing a new vertex and by making that adjacent to a clique with 
$k$ vertices in $G$.

\bigskip
Note that $k$-trees have a decomposition tree of the following form. 
It is a rooted tree with points colored from a set of $k+1$ colors. 
The first $k+1$ vertices closest to the root form a simple path and 
all the points on this path have different colors. 
The rest of the tree branches arbitrarily and there is no restriction 
on the coloring of the points. 

Note that we can assign a unique $(k+1)$-clique to each point in the tree 
as follows. For the points that are in the 
simple path attached to the root it is a $(k+1)$-clique on the points 
of the path. For any other point $x$, the clique is the same as that of 
its parent, except that the point that has the same color as $x$ is 
replaced by $x$.    

It follows that a decomposition tree provides a proper coloring 
of the vertices of a $k$-tree with $k+1$ colors, that is, 
no two adjacent vertices receive the same color. By the recursive 
definition of a $k$-tree, this coloring is unique up to a permutation 
of the colors.  

\begin{lemma}
Let $G=(V,E)$ be a $k$-tree. Let $T$ be a decomposition 
tree for $G$.   
Then for any subset $S \subseteq \{1,\ldots,k+1\}$ 
the graph induced by the vertices with a color in $S$ 
is an $(|S|-1)$-tree. 
\end{lemma}
\begin{proof}
Consider a vertex $x$ with color $i$. Consider the path $P$ from $x$ to 
the root. Then, for any $j \neq i$, 
the vertex $x$ is adjacent to that vertex with color $j \neq i$ 
on $P$ that is furthest from the root. 

Let $S$ be a subset of the colors. Construct a $(|S|-1)$-tree as follows. 
The root-clique is the subset of the root-clique of $T$ with colors 
in $S$. For any vertex with a color in $S$, make it adjacent to the 
vertex on its path to the root in $T$ that is furthest from the root 
and that has a color in $S$. 
\qed\end{proof}

\begin{remark}
Note that in a decomposition tree of a sun-free $k$-tree,  
for every node the children are colored with at most two 
different colors. This holds true for the decomposition tree 
induced by any subset of the colors. 
\end{remark}

\bigskip
The following lemma is basically the same as Theorem~\ref{dahlhaus}. 
We include it because it eases the description of the following 
decomposition of chordal bipartite graphs. 
    
\begin{lemma}[\cite{kn:farber,kn:ho,kn:lehel}]%
\label{G* is SC}
Let $G=(A,B,E)$ be chordal bipartite. Let $G_B^{\ast}$ be the graph 
obtained from $G$ by making a clique of every neighborhood of a 
vertex in $A$. Then $G_B^{\ast}$ is strongly chordal. 
\end{lemma}
\begin{proof}
We write $N^{\ast}(x)$ for the neighborhood of a vertex $x$ in 
$G_B^{\ast}$. 

Consider a vertex $x \in A$ which is simple in $G_B$. 
We claim that $x$ is simple in $G_B^{\ast}$. Assume this is not the case. 
Then there exist two vertices $p$ and $q$ in $N(x)$ that have `private 
neighbors' $p^{\prime}$ and $q^{\prime}$ in $G_B^{\ast}$. That is, 
\[p^{\prime} \in N^{\ast}(p)-N^{\ast}(q) \quad\text{and}\quad 
q^{\prime} \in N^{\ast}(q) -N^{\ast}(p).\]   
Assume that $p^{\prime} \in A$ 
and that $q^{\prime} \in B$. Since $q^{\prime}$ is adjacent to $q$ in 
$G_B^{\ast}$ there exists a vertrex $q^{\prime\prime} \in A$ 
which is adjacent to $q$ and $q^{\prime}$. Note that 
$q^{\prime\prime} \neq p^{\prime}$ since $p^{\prime}$ is not adjacent to 
$q$. Also, $q^{\prime\prime}$ is 
not adjacent to $p$ since $q^{\prime}$ is not adjacent to $p$. 
Then $x$ is not simple in $G_B$. The case where 
$p^{\prime}$ and $q^{\prime}$ are both in $B$ is similar.  

This proves that $G_B^{\ast}$ has a simple elimination ordering 
of the vertices in $A$. 
Let $x \in A$ be simple in $G_B^{\ast}$. If some vertex 
$y \in N(x)$ has no other neighbors in $A$ than $x$, then 
it is simple in $G_B^{\ast}-x$. Thus we obtain an augmented 
simple elimination ordering of all the vertices in $G_B^{\ast}$. 

This proves the lemma.     
\qed\end{proof}

{F}rom~\cite{kn:anstee} we have the following embedding theorem. 
An easier proof of this is described in~\cite{kn:lehel}.  

\begin{theorem}[\cite{kn:anstee,kn:lehel}]
\label{anstee}
Let $G=(A,B,E)$ be chordal bipartite and let $\omega+1$ be the 
maximal cardinality of a clique in the graph $G_B^{\ast}-A$. 
There exists a sequence $T_0,\ldots,T_{\omega}$ such that, 
for $i=0,\ldots,\omega$ the following holds.  
\begin{enumerate}[\rm 1.]
\item $T_0$ is a spanning tree on the vertices 
of $B$ such that for each vertex $x \in A$ with $N(x) \neq \es$, 
the set $N(x)$ 
induces a subtree of $T_0$; 
\item For $i=1 \geq 1$, $T_i=(B,E_i)$ is a strongly chordal 
$i$-tree and it spans the vertices of $B$; 
\item For $i \geq 1$, each $(i+1)$-clique in $T_i$ is the union of two 
$i$-cliques in $T_{i-1}$; 
\item For each vertex $x \in A$ with $N(x) \neq \es$, the set $N(x)$ 
appears as a maximal clique in one of the $T_i$'s. 
\end{enumerate}
\end{theorem}

Consider the $0/1$-adjacency matrix $Q$ with the rows 
indexed by the vertices in $A$ and the columns indexed by the vertices 
in $B$. 
Note that there may be multiple 
copies of identical rows in $Q$; the model does not reflect this fact.  
For any two rows in $Q$ we may add the intersection if it is not 
present already; 
the new matrix is the incidence matrix of a chordal bipartite graph. 
Possibly some rows in $Q$ correspond to $k$-cliques that are 
contained only in one $(k+1)$-clique 
(see, {\em e.g.\/},~\cite{kn:hansen,kn:marcu}). 
We can extend the $k$-trees with additional 
cliques such that each $k$-tree spans all vertices of $B$. 

\bigskip
We describe the data structure that we use in the next section. 
Consider a 
strongly chordal 
$(k-1)$-tree $T_{k-1}$. Consider the maximal clique-tree $T$ for $T_{k-1}$, 
obtained as described at the start of this section. The $k$-tree $T_k$ 
is obtained by a procedure which can be described as 
follows~\cite{kn:lehel}. 
Consider the linegraph $L$ of $T$; thus $L$ is a claw-free blockgraph. 
Let $T^{\prime}$ be a spanning tree of $L$. Each vertex of $T^{\prime}$ 
is a $(k+1)$-clique, which is the union of the two $k$-cliques that are the 
endpoints of the corresponding line in $T$. 
 
\bigskip  
Let $H$ be a $k$-tree on $n$ vertices. Then all maximal cliques 
in $H$ have $k+1$ vertices and there are  
$n-k$ of them. This follows by a simple induction 
from the recursive definition 
of $k$-trees.  
It follows from Theorem~\ref{anstee} 
that a chordal bipartite graph $G=(A,B,E)$ can be embedded into a 
maximal chordal bipartite graph $G^{\prime}=(A^{\prime},B,E)$ with 
\[|A^{\prime}|^{\ast} = \sum_{k=0}^{|B|} (|B|-k) = \sum_{k=0}^{|B|} k
= \binom{|B|+1}{2},\]
where $|A^{\prime}|^{\ast}$ does not take into account 
multiple copies of vertices in $A$ with the same 
neighborhood~\cite{kn:anstee,kn:lehel}. Here we 
assume that the matrix $A$ has no row that contains only zeros (otherwise 
1 should be added to the formula above).  

Let $G=(A,B,E)$ be chordal bipartite. A maximal embedding 
of $G$ is a chordal bipartite graph $G^{\prime}$ obtained as described 
above. A chordal bipartite graph is maximal if its bipartite 
adjacency matrix has a maximal number of different rows.  

\section{Feedback vertex set on chordal bipartite graphs}

\begin{lemma}
\label{basic}
Let $G=(A,B,E)$ be complete bipartite. 
Let $F$ be a feedback vertex set of $G$. Then 
\[ |A - F| \leq 1 \quad\text{or}\quad |B -F| \leq 1.\]
\end{lemma}
\begin{proof}
If $A$ and $B$ both have two vertices that are not in $F$ then 
$G-F$ has an induced 4-cycle. 
\qed\end{proof}

\begin{definition}
Let $G=(A,B,E)$ be a bipartite graph. A {\em hyperedge\/} is 
the neighborhood of a vertex in $A$. 
\end{definition}

\noindent
Note that different vertices in $A$ may define the same hyperedge. 

\begin{lemma}
\label{disjoint}
Let $G=(A,B,E)$ be chordal bipartite and let $R$ be a hyperedge. 
Assume that there exist hyperedges $A$ and $B$ such that 
\begin{enumerate}[\rm (i)]
\item $|A \cap R|=|B \cap R|=1$, and 
\item $A \cap R \neq B \cap R$. 
\end{enumerate}
Then $A \cap B =\es$. 
\end{lemma}
\begin{proof}
To prove this we use the following characterization   
of chordal bipartite graphs 
(see, {\em e.g.\/},~\cite[Proposition~3]{kn:acharya}). 
Let $G=(A,B,E)$ be a bipartite graph. Then $G$ is 
chordal bipartite if and only if   
for any three vertices $x$, $y$ and $z$ in $A$  
at least one of $N(x)$, $N(y)$ and $N(z)$ contains the 
intersection of the other two. 

This characterization proves the lemma.
\qed\end{proof}

\begin{theorem}
There exists a polynomial-time algorithm that solves the feedback vertex 
set problem on chordal bipartite graphs. 
\end{theorem}
\begin{proof}
Let $G=(A,B,E)$ be chordal 
bipartite. We describe the algorithm for an 
embedding of $G$ into a maximal chordal bipartite graph 
$M=(A^{\prime},B,E)$. 
Each hyperedge, defined as the neighborhood of a 
vertex $x$ in $A^{\prime}$, 
has a multiplicity, which is either zero if $x \not\in A$  
or else it is 
the number of vertices $x^{\prime}$ with 
\[N(x^{\prime})=N(x).\]  

We decompose $M_B^{\ast}$ as follows. 
First consider the hyperedges with a maximal number of vertices. 
Say these hyperedges have $k+1$ vertices. The subgraph of 
$M_B^{\ast}-A$ is a $k$-tree $H_k$ and the maximal cliques 
are the hyperedges with $k+1$ vertices. Decompose $H_k$ as described 
in Section~\ref{maximal CBG}. 

Next consider the hyperedges with $k$ vertices. These define a 
$(k-1)$-tree. For each maximal clique $C$ in $H_k$ consider the 
vertices $V_C$ in the subtree rooted at $C$. Decompose the subgraph 
of $M_B^{\ast}-A$ induced by $V_C$ into a $(k-1)$-tree, 
rooted at $C$. 

Continue this decomposition using the $i$-trees 
for $i=k,\ldots,1$. 

Let $C$ be a hyperedge and let $k+1=|C|$. 
The subtree $G_C$ of $C$ is the collection of hyperedges 
contained in the graph induced by the vertices in the subtree of the 
$k$-tree rooted at $C$ (including $C$). Note that some of these 
hyperedges may have cardinality larger than $k+1$. Let $A_C$ denote 
the vertices $x$ in $A$ such that $N(x)$ is a hyperedge in the 
subtree rooted at $C$.   

\bigskip
Consider a hyperedge $C=N(x)$ for some $x \in A$. 
We consider the following cases for the shape of a maximal forest 
in the graph 
induced by the vertices in the subtree rooted at $C$ and prove that 
the maximal cardinality can be determined in polynomial time. 

First assume that 
a maximum forest $T$ contains the vertex $x$ and no other copies of 
$x$. 
If a subtree 
of $T$ that contains $x$ has vertices outside $C$ then it contains a 
vertex $y \in A_C$ such that $|N(y) \cap C|=1$. 
Consider the vertices $b \in C$ and let $Q_b$ be the 
set of vertices $y \in A_C$ 
with  $N(y) \cap C=\{b\}$. 
Consider a maximal subtrees of $C$ with a  
root $C^{\prime}$, 
such that for all  
vertices $z \in A_{C^{\prime}}$ 
the following holds. 
\[|N(z) \cap C| \leq 1 \quad\text{and}\quad 
\left(|N(z) \cap C|=1 \quad\Rightarrow\quad  N(z) \cap C=\{b\}\right).\]     
Find the maximal forest in the subtree rooted 
at $C^{\prime}$ by a table look-up. 
By Lemma~\ref{disjoint}, we may add up the cardinalities of all 
these forests to find a maximal forest rooted at $x$. 

Assume that the multiplicity of $x$ is more than one. Consider a 
forest that contains more than one copy of $x$. Then it contains 
at most one vertex in $C$. Let $b \in C$ and assume that the forest 
contains the vertex $b$. Let $C^{\prime}$ be the root of the 
maximal subtree as defined above. By table look-up we may find the 
maximum forest in the subtree rooted at $C^{\prime}$. 

The vertices in $C-b$ are in the feedback vertex set. To find the 
maximum forest in the subgraph induced by the vertices in the 
subtree rooted at $C$ that are not in the subtree rooted at $C^{\prime}$ 
we proceed as follows. Remove the vertices of $C-b$. If we remove 
the corresponding 
columns from the maximal bipartite adjacency matrix, the new matrix is,  
obviously, totally balanced and the intersection of 
any two rows is a row in the reduced matrix. It follows 
that we can use the same data structure, except that the 
multiplicities of some of the hyperedges changes. 
If a hyperedge contains a vertex that is in the feedback vertex set, 
then the multiplicity of that hyperedge becomes zero. 
Consider a hyperedge that does not contain any vertices in the 
feedback vertex set. Increase the multiplicity by the number of 
vertices $x \in A$ for which the reduced neighborhood 
is exactly that hyperedge. Some vertices $x \in A$ may have their 
neighborhood contained in the feedback vertex set. That is, 
the row of $x$ in the 
reduced matrix becomes zero. These vertices become isolated 
and can be added to any maximal forest.  
Rerun the dynamic programming 
algorithm for the subtrees rooted at $C$ with the new multiplicities,  
except for the subtree rooted at 
$C^{\prime}$. Add up the cardinalities of these maximum forests 
and maximize over the choices of $b \in C$.   

For hyperedges in the subtree of $C$ we update the maximum forest 
in the same manner. Let $Q$ be a hyperedge in the subtree rooted at 
$C$ and let $q \in A$ be such that $N(q)=Q$. 
Some hyperedges in the subtree of $C$ that 
are not in the subtree rooted at $Q$, intersect $Q$ in one vertex. 
Update the maximal cardinality of a maximal forest that contains $q$ 
as described above.      

Consider the case where a maximum forest does not contain $x$. 
Decrease the multiplicity of $C$ by one and update the 
values for all hyperedges in the subtree of $C$ as described above.  
Let $|C|=k+1$. Then $C$ contains exactly two hyperedges $C_1$ and 
$C_2$ of 
cardinality $k$. Let $C_1^{\prime}$ be a hyperedge of maximal 
cardinality in the subtree rooted at $C_1^{\prime}$ and consider 
the decomposition tree rooted at $C_1^{\prime}$ induced by the 
hyperedges contained in the subtree of $C$.   

Note that the number of calls 
to each hyperedge from one of its 
ancestors is bounded by a fixed polynomial.  
This proves that the algorithm terminates in a polynomial 
number of steps. 
\qed\end{proof} 
      
\begin{lemma}
\label{basic 2}
Let $G=(A,B,E)$ be a bipartite graph and let $x$ and $y$ be two 
vertices in $A$. Assume that $N(x) \subseteq N(y)$. 
Assume that there exists a feedback vertex set $F$ in $G$ 
with $x \in F$ and $y \not\in F$. Let 
\[F^{\prime}=(F-x)+y.\] 
Then $F^{\prime}$ is also a feedback vertex set. 
\end{lemma}
\begin{proof}
Assume not. Let $C$ be an induced cycle in $G-F^{\prime}$. 
Then $x \in C$ otherwise $C$ is an induced cycle 
in $G-F$. Let 
\[C^{\prime}=(C-x)+y.\] 
Then $C^{\prime}$ is a cycle in $G-F$ and 
obviously, $C^{\prime}$ contains an 
induced cycle in $G-F$. This is a contradiction. 
\qed\end{proof}

Let $G=(A,B,E)$ be a chordal 
bipartite graph and let $x \in A$ be a simple vertex in 
$G_B$. Let 
$[x_1,\ldots,x_{\ell}]$ 
be an ordering of $N_G(x)$ such that 
\[\text{for $i=1,\ldots,\ell-1$,} \quad N_G(x_i) \subseteq N_G(x_{i+1}).\] 
It follows from Lemma~\ref{basic 2} 
that there is a minimum feedback vertex set $F$ 
such that, for some threshold $t \in \{0,\ldots,\ell\}$, 
\[\forall_{1 \leq i \leq \ell} \;\;
x_i \in F \quad\text{if and only if}\quad i \leq t.\]

\section{Concluding remarks}

Define the chordality of a graph $G$ as the length of a longest 
induced cycle in $G$. We are not aware of any class of graphs of 
bounded chordality on which the feedback vertex set problem 
is NP-complete.

\end{document}